\documentclass{article}
\usepackage{amsmath}
\usepackage{amsfonts}
\usepackage{amssymb}
\usepackage{amsthm}
\usepackage{color}
\usepackage{enumerate}
\usepackage{geometry}
\usepackage{listings}
\usepackage{microtype}
\usepackage{rotating}
\usepackage{textcomp}
\usepackage{makeidx}
\usepackage{framed}
\usepackage{array}
\usepackage[hidelinks]{hyperref}
\usepackage[nameinlink,capitalise]{cleveref}
\usepackage{authblk}
\usepackage{enumitem}
\usepackage{comment}
\usepackage[small]{titlesec}

\newcommand{\N}{\mathbb{N}}

\newcommand{\R}{\mathbb{R}}

\renewcommand{\a}{\alpha}

\newcommand{\e}{\varepsilon}

\newcommand{\D}{\Delta}

\newcommand{\dx}{\,dx}
\newcommand{\dy}{\,dy}

\newcommand{\sint}{\int_{[0,1]^2}}

\renewcommand{\P}{\mathbb{P}}

\newcommand{\limn}{\lim_{n\rightarrow\infty}}

\newcommand{\W}{\mathcal{W}}
\newcommand{\Wt}{\widetilde{\mathcal{W}}}
\newcommand{\M}{\mathcal{M}}
\newcommand{\ds}{d_\square}
\newcommand{\des}{\delta_\square}
\newcommand{\h}{\widetilde{h}}
\newcommand{\g}{\widetilde{g}}

\newcommand{\Ptr}{\widetilde{\P}_{n,r_n}}
\newcommand{\Rs}{\mathcal{R}}
\newcommand{\f}{\widetilde{f}}
\newcommand{\ro}{\overline{r}}
\newcommand{\T}{\mathcal{T}}

\theoremstyle{plain} 
\newtheorem{theorem}{Theorem}[section]
\newtheorem*{theorem*}{Theorem}

\newtheorem*{proposition*}{Proposition}
\newtheorem{lemma}[theorem]{Lemma}
\newtheorem*{lemma*}{Lemma}

\theoremstyle{definition} 
\newtheorem{corollary}[theorem]{Corollary}
\newtheorem*{corollary*}{Corollary}

\newtheorem*{remark*}{Remark}

\newtheorem*{fact*}{Fact}

\newtheorem*{definition*}{Definition}

\newtheorem*{example*}{Example}

\newtheorem*{idea*}{Idea}

\newtheorem*{property*}{Property}

\numberwithin{equation}{section}

\usepackage[english]{babel}

\begin{document}

\title{The large deviation principle for inhomogeneous Erd\H{o}s-R\'enyi random graphs}
\date{\vspace{-5ex}}

\author[1]{Maarten Markering}
\affil[1]{Mathematical Institute, Leiden University} 
\maketitle  
\let\thefootnote\relax\footnotetext{Date: \today}
\let\thefootnote\relax\footnotetext{Email: maartenmarkering@outlook.com}
\let\thefootnote\relax
\let\thefootnote\relax
\let\thefootnote\relax\footnotetext{Data sharing not applicable to this article as no datasets were generated or analysed during the current study.}

\begin{abstract}
Consider the \textit{inhomogeneous Erd\H{o}s-R\'enyi random graph} (ERRG) on $n$ vertices for which each pair $i,j\in\{1,\ldots,n\}$, $i\neq j$ is connected independently by an edge with probability $r_n(\frac{i-1}{n},\frac{j-1}{n})$, where $(r_n)_{n\in\N}$ is a sequence of graphons converging to a \textit{reference graphon} $r$. As a generalization of the celebrated large deviation principle (LDP) for ERRGs by Chatterjee and Varadhan \cite{chatterjee2010large}, Dhara and Sen \cite{dhara2019large} proved an LDP for a sequence of such graphs under the assumption that $r$ is bounded away from 0 and 1, and with a rate function in the form of a \textit{lower semi-continuous envelope}. We further extend the results by Dhara and Sen. We relax the conditions on the reference graphon to $\log r,\log(1-r)\in L^1([0,1]^2)$. We also show that, under this condition, their rate function equals a different, more tractable rate function. We then apply these results to the large deviation principle for the \textit{largest eigenvalue} of inhomogeneous ERRGs and weaken the conditions for part of the analysis of the rate function by Chakrabarty, Hazra, Den Hollander and Sfragara \cite{chakrabarty2020large}.
\end{abstract}

\section{Introduction}
\subsection{Motivation and outline}
The large deviation principle (LDP) for the Erd\H{o}s-R\'enyi random graph (ERRG) was introduced and proved by Chatterjee and Varadhan in their seminal paper \cite{chatterjee2010large}. They viewed the sequence of ERRGs as \textit{graphons} and obtained the LDP in the space of graphons with the cut topology. With this LDP, Chatterjee and Varadhan completely solved a long-standing open problem regarding upper tails for large deviations of triangle counts in ERRGs. This spurred many further developments in the area of large deviations for random graphs. We refer to \cite{Chatterjeebook} for an overview.

In this paper, we study the \textit{inhomogeneous Erd\H{o}s-R\'enyi random graph}, which is a generalization of the ERRG in that different edges do not necessarily occur with the same probability. This probability is controlled by a \textit{reference graphon}. Recently, Dhara and Sen \cite[Proposition 3.1]{dhara2019large} proved the LDP for the inhomogeneous ERRG model under the assumption that the reference graphon $r$ is bounded away from 0 and 1, i.e., there exists $\eta>0$ such that $\eta\leq r(x,y)\leq 1-\eta$ for all $(x,y)\in[0,1]^2$. The rate function $J'_r$ for their LDP has the form of a \textit{lower semi-continuous envelope} of another rate function $I_r$, which complicates its analysis.

We extend their proof to reference graphons that satisfy the mild integrability condition $\log r,\log(1-r)\in L^1([0,1]^2)$. Furthemore, we show that $J'_r$ is also the lower semi-continuous envelope of another, more tractable rate function $J_r$ that is already semi-continuous, which implies that $J'_r$ actually equals $J_r$. The relaxation of the conditions broadens the scope of applications, and the simplification of the rate function makes the LDP more suitable to such applications. As an example application, we consider the LDP for the \textit{largest eigenvalue} of an inhomogeneous ERRG. We show that the conditions for the analysis by Chakrabarty, Hazra, Den Hollander and Sfragara \cite{chakrabarty2020large} can partially be weakened.

Random graphs with \textit{inhomogeneities} and \textit{constraints} have many applications in complex networks, physics and statistics. As a consequence, recent interest in large deviations for inhomogeneous random graphs has grown considerably. The LDP for ERRGs was applied by Chatterjee and Diaconis \cite{chatterjee2013} to the exponential random graph, Dhara and Sen \cite{dhara2019large} applied the LDP for inhomogeneous ERRGs to random graphs with fixed degrees, and recently, Borgs et al. \cite{borgs2020large} proved an LDP for block models. This paper is part of a general effort to better understand large deviations for inhomogeneous random graphs.

\paragraph{Outline.} In Section \ref{sec:graphons}, we briefly introduce the necessary concepts and definitions from graph limit theory. The LDP for inhomogeneous ERRGs is stated in Section \ref{sec:maintheorem}, and the rate function is introduced. In Section \ref{section:largesteigenvalueintro}, we introduce large deviations for the largest eigenvalue of the inhomogeneous ERRG. The proof that the rate function is well-defined, lower semi-continuous, and equal to the rate function of Dhara and Sen \cite{dhara2019large} is given in Section \ref{section:ratefunction}. We generalize Dhara and Sen's proof of the LDP upper bound in Section \ref{section:upperbound}. In Section \ref{section:largesteigenvalue}, we finish by showing that the results from this paper can be used to weaken the conditions of the analysis of the rate function for the largest eigenvalue by Chakrabarty et al. \cite{chakrabarty2020large}.

\subsection{Graphons}\label{sec:graphons}
A \textit{graphon} is a Borel measurable function $h\colon[0,1]^2\to[0,1]$ such that $h(x,y)=h(y,x)$ for all $(x,y)\in[0,1]^2$. We denote the set of graphons by $\W$. Every finite simple graph $G=(V(G),E(G))$ with $V(G)=[n]:=\{1,\ldots,n\}$ can be represented as the graphon $h^G$ defined as
\begin{equation}
    h^G(x,y)=\begin{cases}
    1, & (i,j)\in E(G),\,(x,y)\in B(i,j,n),\\
    0, &\text{otherwise,}
    \end{cases}
\end{equation}
with $B(i,j,n):=[\frac{i-1}{n},\frac{i}{n})\times[\frac{j-1}{n},\frac{j}{n})$. We call $h^G$ the \textit{empirical graphon} of $G$. Let $\M$ denote the set of Lebesgue measure-preserving bijections $\phi\colon[0,1]\to[0,1]$. The \textit{cut distance} on $\W$ is defined as 
\begin{equation}
    \ds(h_1,h_2):=\sup_{S,T\subset[0,1]}\left|\int_{[0,1]^2}(h_1(x,y)-h_2(x,y))\dx\dy\right|
\end{equation}
and the \textit{cut metric} is defined as
\begin{equation}
    \des(h_1,h_2):=\inf_{\phi\in\M}\ds(h^\phi_1,h_2),
\end{equation}
where $h^\phi(x,y):=h(\phi(x),\phi(y))$. See \cite[Theorem 8.13]{Lovaszbook} for several equivalent definitions for $\des$. The cut metric induces an equivalence relation $\sim$ on $\W$ where $h_1\sim h_2$ if $\des(h_1,h_2)=0$. Define $\Wt:=\W/{\sim}$ and denote the equivalence class of $h\in\W$ by $\h$. The space $(\Wt,\des)$ is a compact metric space \cite[Theorem 9.23]{Lovaszbook}.

\subsection{Main theorem}\label{sec:maintheorem}
For some $h\in\W$, define the random graph $G_n$ with vertex set $[n]$ by connecting every pair of vertices $i,j\in[n]$ with probability $h(\frac{i-1}{n},\frac{j-1}{n})$. Denote the law of the empirical graphon $h^{G_n}$ of $G_n$ on $\W$ by $\P_{n,h}$ and the law of $\widetilde{h}^{G_n}$ on $\Wt$ by $\widetilde{\P}_{n,h}$. 

We define a sequence of random graphs as follows. Fix a graphon $r\in\W$ called the \textit{reference graphon}
and let $(r_n)_{n\in\N}$ be a sequence of \textit{block graphons} of the form
\begin{equation}\label{eq:blockgraphon}
    r_n=\begin{cases}
    r_{n,ij}, &(x,y)\in B(i,j,n),\, 1\leq i,j\leq n,\, i\neq j,\\
    0, &\text{otherwise},
    \end{cases}
\end{equation}
such that $0<r<1$ and $r_n\rightarrow r$ Lebesgue-almost everywhere and in $L^1$-norm as $n\rightarrow\infty$. Further assume that 
\begin{equation}\label{eq:assreference}
    \log r,\log(1-r)\in L^1([0,1]^2)
\end{equation}
and 
\begin{equation}\label{eq:asssequence}
    \|\log r_n-\log r\|_{L^1},\|\log(1-r_n)-\log(1-r)\|_{L^1}\rightarrow0
\end{equation}
as $n\rightarrow\infty$.

We show that $(\Ptr)_{n\in\N}$ satisfies an LDP. First, we define a suitable rate function. For $a\in[0,1]$ and $b\in(0,1)$, let
\begin{equation}\label{eq:variationrate}
    \begin{split}
        \Rs(a\mid b)&:=a\log\frac{a}{b}+(1-a)\log\frac{1-a}{1-b},
    \end{split}
\end{equation}
where we use the convention $0\log0=0$. For $h,r\in\W$ such that $0<r<1$ Lebesgue-almost everywhere, this map can be extended to a map $I_r$ on $\W$ by defining
\begin{equation}
    I_{r}(h)=\int_{[0,1]^2}\Rs(h(x,y)\mid r(x,y))\dx\dy.
\end{equation}
In the case $r\equiv p$, $I_r$ is invariant under measure-preserving bijections. Hence, $I_r$ can be extended to $\Wt$ as $J_r(\h):=I_r(h)$ \cite[Proposition 5.1]{Chatterjeebook}. For inhomogeneous reference graphons, this is no longer the case. To solve this problem, Dhara and Sen extend the function to its \textit{lower semi-continuous envelope}, i.e.,
\begin{equation}\label{eq:ratedharasen}
    J_r'(\h):=\sup_{\eta>0}\inf_{g\in B(\h,\eta)}I_r(g),
\end{equation}
where $B(\h,\eta):=\{g\in\W\mid\des(\h,\g)\leq\eta\}$. This  extension is well-defined and lower semi-continuous \cite[Lemma 2.1]{dhara2019large}, but analytic manipulation can be somewhat difficult. Instead, we propose the following more tractable rate function:
\begin{equation}
    J_r(\h):=\inf_{\phi\in\M}I_r(h^\phi).
\end{equation}
A priori, it is not clear whether $J_r$ a good rate function on $\Wt$ or even well-defined. In Section \ref{section:ratefunction}, we show that it is, and that $J_r$ in fact equals $J'_r$ under the condition \eqref{eq:assreference}. This is one of the main results of this paper.

We are now ready to state the main theorem.
\begin{theorem}\label{th:maintheorem}
Subject to \eqref{eq:assreference} and \eqref{eq:asssequence}, the sequence $(\Ptr)_{n\in\N}$ satisfies the large deviation principle on $(\Wt,\des)$ with rate $\frac{n^2}{2}$ and rate function $J_r$, i.e., for all closed sets $\widetilde{F}\subset\Wt$ and open sets $\widetilde{U}\subset\Wt$,
\begin{equation}
    \begin{split}
        \limsup_{n\rightarrow\infty}\frac{2}{n^2}\log\Ptr(\widetilde{F})\leq-\inf_{\h\in\widetilde{F}}J_r(\h),\\
        \liminf_{n\rightarrow\infty}\frac{2}{n^2}\log\Ptr(\widetilde{U})\geq-\inf_{\h\in\widetilde{U}}J_r(\h).
    \end{split}
\end{equation}
\end{theorem}
This theorem was proved by Dhara and Sen \cite[Proposition 3.1]{dhara2019large} under the condition that there exists an $\eta>0$ such that $\eta\leq r(x,y)\leq 1-\eta$ and $\eta\leq r_{n,ij}(x,y)\leq1-\eta$ for all $(x,y)\in[0,1]^2$, $n\in\N$ and $i\neq j$. The novelty in this paper lies in weakening the conditions and showing that $J_r=J'_r$.

The proof of the lower bound requires only minor adjustments of the proof in \cite{Chatterjeebook}. Therefore, we only prove the upper bound in this paper. For a detailed proof of the lower bound, we refer to \cite[Section 6]{Markeringthesis}. Throughout the rest of the paper, we implicitly assume that \eqref{eq:assreference} and \eqref{eq:asssequence} hold and no longer mention it in the statement of our results.

\subsection{Large deviations for the largest eigenvalue}\label{section:largesteigenvalueintro}
Let $\lambda_n$ be the largest eigenvalue of the adjacency matrix of $G_n$. Then $\frac{\lambda_n}{n}$ also satisfies an LDP. Chakrabarty et al. \cite{chakrabarty2020large} studied the rate function $\psi_r$ under the conditions that $r$ is bounded away from 0 and 1 and of rank 1, i.e. $r(x,y)=\nu(x)\nu(y)$ for some $\nu\colon[0,1]\to[0,1]$. They analysed the scaling of the rate function and identified the form of the minimisers near the rate function's minimum and near the boundaries 0 and 1.

The requirement that $r$ is bounded away from 0 and 1 stems in part from the fact that the results from \cite{chakrabarty2020large} are obtained using the LDP by Dhara and Sen \cite{dhara2019large}. They posed the question whether the boundedness condition could be weakened to some form of integrability condition. In this paper, we show that the condition can partially be relaxed to \eqref{eq:assreference}. In particular, we extend their analysis of the rate function near its minimum.

From the LDP for inhomogeneous ERRGs, Chakrabarty et al. \cite{chakrabarty2020large} derive the following LDP for the largest eigenvalue. Note that for this theorem, we do not yet require $r$ to be of rank 1.
\begin{theorem}
Let $\P^*_n$ denote the law of $\lambda_n/n$. Subject to \eqref{eq:assreference} and \eqref{eq:asssequence}, the sequence $(\P^*_n)_n\in\N$ satisfies the LDP on $\R$ with rate $n^2/2$ and with rate function
\begin{equation}\label{eq:rateeigenvalue}
    \begin{split}
        \psi_r(\beta)=\inf_{\substack{\h\in\Wt\\\|T_h\|=\beta}}J_r(\h)=\inf_{\substack{h\in\W\\\|T_h\|=\beta}}I_r(h),\quad\beta\in\R,
    \end{split}
\end{equation}
with $T_h$ the operator on $L^2([0,1]^2)$ defined as
\begin{equation}
    \begin{split}
        T_h(u)(x)=\int_{[0,1]}h(x,y)u(y)\dy
    \end{split}
\end{equation}
for $h\in\W$, $u\in L^2([0,1]^2)$ and $x\in[0,1]$, and where $\|T_h\|$ is the operator norm of $T_h$ with respect to the $L^2$-norm on $L^2([0,1]^2)$.
\end{theorem}
\begin{proof}
The proof is standard and follows from Theorem \ref{th:maintheorem}, combined with the observation that $\lambda_n/n=\|T_{h^{G^n}}\|$. See \cite[Theorem 1.4]{chakrabarty2020large}.
\end{proof}
Note that Chakrabarty et al. \cite{chakrabarty2020large} already use the rate function $J_r$, so their result is already an application of the results from this paper.

Put $C_r=\|T_r\|$. Then Chakrabarty et al. \cite{chakrabarty2020large} show that the rate function $\psi_r$ is continuous and unimodal on $[0,1]$, with a unique zero at $C_r$, and that it is strictly decreasing and strictly increasing on $[0,C_r]$ and $[C_r,1]$ respectively. Furthermore, for every $\beta\in[0,1]$, the set of minimisers of the variational formula for $\psi_r(\beta)$ in \eqref{eq:rateeigenvalue} is non-empty and compact in $\Wt$. For $\beta\not\in[0,1]$, $\psi_r(\beta)=+\infty$. These results do not require boundedness away from 0 and 1 of the reference graphon.

One of the main results by Chakrabarty et al. \cite{chakrabarty2020large} is the scaling of the rate function and the minimisers near $\beta=C_r$ under the condition that $r$ is bounded away from 0 and 1. We generalize it to the following theorem, which we prove in Section \ref{section:largesteigenvalue}.
\begin{theorem}\label{th:scalingeigenvalue}
Assume there exists $\nu\colon[0,1]\to[0,1]$ such that $r(x,y)=\nu(x)\nu(y)$. Then, subject to \eqref{eq:assreference},
\begin{equation}
    \psi_r(\beta)=(1+o(1))K_r(\beta-C_r)^2,\quad\beta\rightarrow C_r,
\end{equation}
with
\begin{equation}
    K_r=\frac{C_r^2}{2B_r},
\end{equation}
where
\begin{equation}\label{eq:Br}
    B_r=\int_{[0,1]^2}r(x,y)^3(1-r(x,y))\dx\dy.
\end{equation}
Furthermore, let $h_\beta\in\W$ be any minimiser of the second infimum in \eqref{eq:rateeigenvalue}. Then
\begin{equation}
    \lim_{\beta\rightarrow C_r}(\beta-C_r)^{-1}\|h_\beta-r-(\beta-C_r)\Delta\|_{L^2}=0,
\end{equation}
with
\begin{equation}
    \Delta=\frac{C_r}{B_r}r^2(1-r).
\end{equation}
\end{theorem}

\subsection{Discussion}
\paragraph{Conditions on the reference graphon.}

In \cite{borgs2020large}, Borgs et al. proved an LDP for a block model in which the reference graphon consists of rational-length blocks. This result was later strengthened to include blocks of arbitrary length by Greb\'ik and Pikhurko \cite{grebik2021large}. In their model, the reference graphon may take on the values 0 or 1. It would be interesting to generalise both LDPs to a single LDP for graphons that are partly block graphons (that can attain the values 0 and 1), and partly satisfy the integrability condition of this paper. Since the condition $\log r,\log(1-r)\in L^1([0,1]^2)$ pervades almost every step of the proof of Theorem \ref{th:maintheorem}, it appears to be difficult to obtain an LDP for \emph{arbitrary} reference graphons. It might be the case that the inhomogeneous ERRG model does not satisfy an LDP for certain reference graphons.

\paragraph{Non-dense random graphs.}
Graph limit theory provides the right framework for studying sequences of dense graphs, since non-dense graphs converge to the zero graphon. A similar framework for non-dense graphs is still in development, so not much is known yet for large deviations of non-dense random graphs. We refer to the bibliographical notes in \cite[Chapter 6]{Chatterjeebook} and \cite[Section 3]{chatterjee2016introduction} for a short review of recent results in sparse graph limit theory and sparse large deviations. Although this paper also considers dense random graphs, the reference graphon is allowed to approach 0. Such reference graphons can induce dense random graphs with non-dense subgraphs. This is in contrast to the setting of Dhara and Sen \cite{dhara2019large}, where every subgraph is also dense.

\paragraph{Rate function for the LDP for block models.}
Like Dhara and Sen \cite{dhara2019large}, Borgs et al. \cite{borgs2020large} use a lower semi-continuous envelope as their rate function. The rate function from this paper can also be used in \cite{borgs2020large}. The authors also derive an LDP for homomorphism densities and define a symmetric and symmetry breaking regime. The existence of a symmetry breaking regime is only established for specific block models, in part due to the intractable nature of the rate function. The precise boundary between the symmetric and non-symmetric regimes was also only identified for bipartite ERRGs, again due to the intractability of the rate function. The results from this paper might aid in resolving the general cases.

\paragraph{Acknowledgements}
This work was initiated as a Bachelor thesis at Leiden University under the supervision of Frank den Hollander \cite{Markeringthesis}. The author thanks him for his continued guidance and support the past year. The author is very grateful for the many mathematical discussions, as well as his thorough feedback on the thesis and help with this paper.

The author also thanks an anonymous referee for their careful review.

\section{The rate function}\label{section:ratefunction}
We show that the candidate rate function is a good rate function, i.e., does not equal infinity everywhere and has compact level sets. Since the first requirement is clear and $\Wt$ is compact, it suffices to show that $J_r$ is lower semi-continuous. First, we need to check that $J_r$ is well-defined on the quotient space $\Wt$. As a consequence of lower semi-continuity, we prove that $J_r$ equals the rate function $J'_r$ as defined by Dhara and Sen \cite{dhara2019large} (see \eqref{eq:ratedharasen}).

\begin{lemma}\label{lemma:ratecontinuousL2}
The function $I_r$ is continuous in the $L^2$-topology on $\W$.
\end{lemma}
\begin{proof}
Let $(f_n)_{n\in\N}\subset\W$ and $f\in\W$ such that $(f_n)_{n\in\N}$ converges to $f$ in $L^2([0,1]^2)$. Note that for all $h\in\W$,
\begin{equation}\label{eq:dctbound2}
    \begin{split}
        \left|\Rs(h(x,y)\mid r(x,y))\right|\leq&|h(x,y)\log h(x,y)|+|h(x,y)\log  r(x,y)|\\
        &+|(1-h(x,y))\log(1-h(x,y))|+|(1-h(x,y))\log(1- r(x,y))|\\
        \leq&\frac{2}{e}+|\log  r(x,y)|+|\log(1- r(x,y))|
    \end{split}
\end{equation}
for all $(x,y)\in[0,1]^2$, where we use that $x\mapsto x\log x$ has a minimum $-\frac{1}{e}$ on $[0,1]$. Because $\log  r,\log(1- r)\in L^1([0,1]^2)$, the bound above is integrable.

Let $(n_k)_{k\in\N}$ be a sequence of integers tending to infinity. Since $f_n\rightarrow f$ as $n\rightarrow\infty$ in $L^2([0,1]^2)$, there exists a subsequence $(n_{k_l})_{l\in\N}$ such that $f_{n_{k_l}}\rightarrow f$ as $l\rightarrow\infty$ Lebesgue-almost everywhere. By continuity, $\Rs(f_{n_{k_l}}(x,y)\mid r(x,y))\rightarrow\Rs(f(x,y)\mid r(x,y))$ Lebesgue-almost everywhere. By the dominated convergence theorem, using the bound from \eqref{eq:dctbound2}, we have $  I_r(f_{n_{k_l}})\rightarrow   I_r(f)$ as $l\rightarrow\infty$. Thus, for every sequence $(n_k)_{k\in\N}$ there exists a subsequence $(n_{k_l})_{l\in\N}$ such that $  I_r(f_{n_{k_l}})\rightarrow   I_r(f)$ as $l\rightarrow\infty$. Hence, $  I_r(f_n)\rightarrow   I_r(f)$ as $n\rightarrow\infty$.
\end{proof}

\begin{lemma}\label{lemma:ratewelldefined}
Let $f,g\in\W$ be such that $\delta_\square(f,g)=0$. Then $\inf_{\phi\in\M} I_{r}(f^\phi)=\inf_{\phi\in\M} I_{r}(g^\phi)$.
\end{lemma}
\begin{proof}
Note that by \cite[Corollary 8.14]{Lovaszbook}, $\delta_\square(f,g)=0$ if and only if there exists a sequence $(\phi_n)_{n\in\N}\subset\M$ such that $\|f ^{\phi_n}-g\|_{L^2}\rightarrow0$ as $n\rightarrow\infty$. Since $\|h_1^\phi-h_2^\phi\|_{L^2}=\|h_1-h_2\|_{L^2}$ for all $h_1,h_2\in\W$ and $\phi\in\M$, we find that $\|(f ^{\phi_n}) ^\phi-g ^\phi\|_{L^2}\rightarrow0$ for all $\phi\in\M$. Thus, by Lemma \ref{lemma:ratecontinuousL2}, $ I_r((f ^{\phi_n}) ^\phi)\rightarrow  I_r(g ^\phi)$ as $n\rightarrow\infty$. Because this holds for all $\phi\in\M$, we obtain
\begin{equation}
    \begin{split}
        \inf_{\phi\in\M} I_r(g ^\phi)&=\inf_{\phi\in\M}\limn I_r((f ^{\phi_n})^\phi)\geq\inf_{\phi\in\M} I_r(f ^\phi).
    \end{split}
\end{equation}
By symmetry, the reverse inequality also holds.
\end{proof}
By Lemma \ref{lemma:ratewelldefined}, $J_r$ can also be expressed as 
\begin{equation}\label{eq:expressionrateequivalence}
     J_r(\widetilde{f})=\inf_{g\in B(\widetilde{f},0)} I_r(g),
\end{equation}
so
\begin{equation}
     J_r'(\widetilde{f})=\sup_{\eta>0}\inf_{g\in B(\widetilde{f},\eta)} I_r(g)=\sup_{\eta>0}\inf_{\widetilde{g}\in \widetilde{B}(\widetilde{f},\eta)} J_r(\widetilde{g}).
\end{equation}
Since $J_r\geq J_r'$, we have $J_r=J_r'$ if $J_r$ is lower semi-continuous on $\Wt$.

\begin{theorem}
Subject to \eqref{eq:assreference} and \eqref{eq:asssequence}, the function $ J_r$ is lower semi-continuous on $\Wt$.
\end{theorem}
\begin{proof}
Let $\widetilde{f}\in\Wt$ and $(\f_n)_{n\in\N}\subset\Wt$ such that $\widetilde{f}_n\rightarrow\widetilde{f}$ in $\des$. Without loss of generality, we may assume that $f_n\rightarrow f$ in $d_\square$. Let $\Delta_n:=f_n-f\in\W_1:=\{f-g\mid f,g\in\W\}$. Then $\Delta_n\rightarrow0$ in $d_\square$ and by an easy computation,
 \begin{equation}
    I_{r}(f^\phi+\Delta_n^\phi)=I_p(f^\phi+\Delta^\phi_n)+I_r(f^\phi)-I_p(f^\phi)+\int_{[0,1]^2}\Delta^\phi_n(x,y)\log\left(\frac{1-r(x,y)}{r(x,y)}\frac{p}{1-p}\right)\dx\dy
\end{equation}
for every $n\in\N$, $\phi\in\M$ and $p\in(0,1)$. Thus,
\begin{equation}
    \begin{split}
    &\liminf_{n\rightarrow\infty} J_r(\widetilde{f}_n)=\liminf_{n\rightarrow\infty} J_r(\widetilde{f+\Delta}_n)=\liminf_{n\rightarrow\infty}\inf_{\phi\in\M} I_r(f^\phi+\Delta_{n}^{\phi})\\
    =&\liminf_{n\rightarrow\infty}\inf_{\phi\in\M}\Bigg(I_p(f^\phi+\Delta_{n}^{\phi})+ I_r(f^\phi)-I_p(f^\phi)\\
    &\left.+\int_{[0,1]^2}\Delta_{n}^{\phi}(x,y)\log\left(\frac{1-r(x,y)}{r(x,y)}\frac{p}{1-p}\right)\dx\dy\right)\\
    \geq&\liminf_{n\rightarrow\infty}\left(I_p(f+\Delta_n)-I_p(f)+ J_r(\widetilde{f})+\inf_{\phi\in\M}\int_{[0,1]^2}\Delta_{n}^{\phi}(x,y)\log\left(\frac{1-r(x,y)}{r(x,y)}\frac{p}{1-p}\right)\dx\dy\right)\\
    \geq& J_r(\widetilde{f})+\liminf_{n\rightarrow\infty}\inf_{\phi\in\M}\int_{[0,1]^2}\Delta_{n}^{\phi}(x,y)\log\left(\frac{1-r(x,y)}{r(x,y)}\frac{p}{1-p}\right)\dx\dy\\
    \geq& J_r(\widetilde{f})+\liminf_{n\rightarrow\infty}\int_{[0,1]^2}\Delta_{n}^{\phi_n}(x,y)\log\left(\frac{1-r(x,y)}{r(x,y)}\frac{p}{1-p}\right)\dx\dy-\varepsilon= J_r(\widetilde{f})-\varepsilon
    \end{split}
\end{equation}
for some $p\in(0,1)$, arbitrary $\varepsilon>0$ and some sequence $(\phi_n)_{n\in\N}\subset\M$. The second inequality follows from the lower semi-continuity of $I_p$ on $\W$. The last inequality is obtained by noting that $\Delta_{n}^{\phi_n}\rightarrow0$ in $d_\square$ and $\log r,\log(1-r)\in L^1([0,1]^2)$ and
applying \cite[Lemma 8.22]{Lovaszbook}. Because $\varepsilon>0$ is arbitrary, the proof is complete.
\end{proof}
\begin{corollary}
For all $\f\in\W$, $ J_r(\f)= J_r'(\f)$.
\end{corollary}

\section{The upper bound}\label{section:upperbound}
The proof of the upper bound is an adaptation of the proof by Dhara and Sen \cite[Proposition 3.1]{dhara2019large}. It suffices to prove the following result.
\begin{theorem}\label{th:upperbound}
Let $\e>0$ and $\f\in\Wt$. Then there exists an $\eta(\e)>0$ such that, for all $\eta\in(0,\eta(\e))$,
\begin{equation}
    \limsup_{n\rightarrow\infty}\log\Ptr(\widetilde{B}(\widetilde{f},\eta))\leq-\inf_{g\in B(\widetilde{f},4\e)}I_r(g)+\e,
\end{equation}
with $\widetilde{B}(\f,\eta)=\{\g\in\Wt\mid\des(\f,\g)\leq\eta\}$ and $B(\f,4\e)=\{g\in\W\mid\des(\f,\g)\leq4\e\}$.
\end{theorem}
The proof is done via the \textit{level-$n$ approximants} $(\overline{r}_n)_{n\in\N}$ of $r$, which are of the form \eqref{eq:blockgraphon}, with  
\begin{equation}
    \ro_{n,ij}:=n^2\int_{B(i,j,n)} r(x,y)\dx\dy.
\end{equation}
Dhara and Sen show that the distributions $\frac{1}{n^2}\log\P_{n,r_n}$ are well-approximated by $\frac{1}{n^2}\P_{n,\overline{r}_k}$ for $n$ large enough and some fixed $k$ and that the rate function $I_r$ is well-approximated by $I_{\overline{r}_k}$ for $k$ large enough. In the case that $r$ is bounded away from 0 and 1, Dhara and Sen use Lipschitz continuity of the logarithm on a closed interval. If $r$ tends to 0, $\log\P_{n,r_n}$ and $\log\P_{n,\overline{r}_k}$ might differ by large amounts as $n\rightarrow\infty$. Thus, we require more control over the approximation in this paper. We obtain this by precisely counting the points in the unit square where $\log r_n$ and $\log\overline{r}_k$ are far apart and showing that this area tends to 0 sufficiently fast. The proof is given in Section \ref{section:fixedblockgraphon}. In Section \ref{section:blockapprox}, we show that the rate function induced by the level-$n$ approximants approximates the rate function induced by $r$ well. In Section \ref{ss:upperboundproof}, we finish the proof of Theorem \ref{th:upperbound}. This part of the proof does not require $r$ to be bounded away from 0 and 1. Hence, this section does not contain any original content, but we still include it for the sake of completeness.


\subsection{Block graphon approximants}\label{section:blockapprox}
By \cite[Proposition 2.6]{Chatterjeebook}, the level-$n$ approximants converge to $r$ in $L^1$-norm, and convergence almost everywhere follows from the Lebesgue differentiation theorem for sets of bounded eccentricity \cite[Chapter 3, Corollary 1.7]{stein2005real}. The following lemma shows that the level-$n$ approximants satisfy \eqref{eq:asssequence}. The second lemma is a generalization of \cite[Lemma 2.3]{dhara2019large}.

\begin{lemma}\label{lemma:logconvergenceapprox}
$\|\log\ro_n-\log r\|_{L^1}\rightarrow0$ and $\|\log(1-\ro_n)-\log(1-r)\|_{L^1}\rightarrow0$ as $n\rightarrow\infty$.
\end{lemma}
\begin{proof}
First, note that, for $(x,y)\in B(i,j,n)$,
\begin{equation}
    \begin{split}
        |\log\ro_n(x,y)|=&-\log n^2\int_{B(i,j,n)}r(u,v)\,du\,dv\leq n^2\int_{B(i,j,n)}-\log r(u,v)\,du\,dv.
    \end{split}
\end{equation}
The equality follows from the fact that $0\leq r\leq 1$, and the inequality follows from Jensen's inequality. This upper bound is integrable, since
\begin{equation}
    \begin{split}
        \sum_{1\leq i,j\leq n}n^2\int_{B(i,j,n)}-\log r(u,v)\,du\,dv=\|\log r\|_{L^1}<\infty.
    \end{split}
\end{equation}
By the dominated convergence theorem and the fact that $\ro_n$ converges to $r$ almost everywhere, $\|\log\ro_n-\log r\|_{L^1}\rightarrow0$. The proof that $\|\log(1-\ro_n)-\log(1-r)\|_{L^1}\rightarrow0$ is completely analogous.
\end{proof}

\begin{lemma}\label{lemma:ratecontinuousreference}
Let $r\in\W$ and $(r_n)_{n\in\N}\subset\W$ that satisfy \eqref{eq:assreference} and \eqref{eq:asssequence}. Then $I_{r_n}(\f)\rightarrow I_r(\f)$ uniformly over $f\in\W$ as $n\rightarrow\infty$.
\end{lemma}
\begin{proof}
First note that $\|\log r_n-\log r\|_{L^1}\rightarrow0$ and $\|\log(1-r_n)-\log(1-r)\|_{L^1}\rightarrow0$ by Lemma \ref{lemma:logconvergenceapprox}. Furthermore, for all $f\in \W$,
\begin{equation}
    \begin{split}
        |I_{r_n}(f)-I_r(f)|&=\left|\int_{[0,1]^2}\left(f(x,y)\log\frac{r_n(x,y)}{r(x,y)}+(1-f(x,y))\log\frac{1-r_n(x,y)}{1-r(x,y)}\right)\dx\dy\right|\\
        &\leq\|\log r_n-\log r\|_{L^1}+\|\log(1-r_n)-\log(1-r)\|_{L^1}.
    \end{split}
\end{equation}
Hence, by the definition of the rate function on $\Wt$,
\begin{equation}
    \begin{split}
        \left|J_{r_n}(\f)-J_r(\f)\right|&=\left|\inf_{\phi\in\M}I_{r_n}(f^\phi)-\inf_{\phi\in\M}I_r(f^\phi)\right|\leq\sup_{\phi\in\M}\left|I_{r_n}(f^\phi)-I_r(f^\phi)\right|\\
        &\leq\|\log r_n-\log r\|_{L^1}+\|\log(1-r_n)-\log(1-r)\|_{L^1}.
    \end{split}
\end{equation}
Since the bound is uniform over all $\f\in\Wt$, we obtain the desired result.
\end{proof}

\subsection{Approximation by a fixed block graphon}\label{section:fixedblockgraphon}
The following lemma is a generalization of \cite[Lemma 3.2]{dhara2019large}.
\begin{lemma}\label{lemma:blockapprox}
Let $r\in\W$ and $(r_n)_{n\in\N}\subset\W$ that satisfy \eqref{eq:assreference} and \eqref{eq:asssequence}. For all $\varepsilon>0$ sufficiently small there exist $N_0=N_0(\varepsilon)$ and $N_1=N_1(\varepsilon)$ such that for all $n\geq N_1\geq k\geq N_0$, $f\in\W$ and $\eta>0$,
\begin{equation}
    \left|\frac{1}{n^2}\log\P_{n,r_n}(B(\f,\eta))-\frac{1}{n^2}\log\P_{n,\ro_{k}}(B(\f,\eta))\right|<\varepsilon,
\end{equation}
where $\ro_k$ is the level-$k$ approximant of $r$ as defined in Section \ref{section:blockapprox}.
\end{lemma}
\begin{proof}
Let $\varepsilon>0$. Since $\|\log r_n-\log r\|_{L^1},\|\log \ro_n-\log r\|_{L^1}\rightarrow0$ and $\|\log(1-r_n)-\log(1-r)\|_{L^1},\|\log(1-\ro_n)-\log(1-r)\|_{L^1}\rightarrow0$ by Lemma \ref{lemma:logconvergenceapprox}, there exists an $N_0\in\N$ such that 
$\|\log r_n-\log\ro_k\|_{L^1}<\varepsilon/4$ and $\|\log(1-r_n)-\log(1-\ro_k)\|_{L^1}<\varepsilon/4$ for all $n\geq k\geq N_0$.

Note that 
\begin{equation}\label{eq:probabilitiesasintegral}
    \begin{split}
        \P_{n,r_n}(B(\f,\eta))=\int_{B(\f,\eta)}\exp\left(\log\frac{d\P_{n,r_n}}{d\P_{n,\ro_k}}\right)\,d\P_{n,\ro_k}.
    \end{split}
\end{equation}
Fix $n\geq k\geq N_0$, and let $g\in\W$ be of the form \eqref{eq:blockgraphon}. Denote by $r_{n,uv}$ and $g_{uv}$ the values of $r_n$ and $g$ in $B(u,v,n)$ respectively. Then
\begin{equation}\label{eq:logdifferenceestimation}
    \begin{split}
        &\frac{1}{n^2}\left|\log\frac{d\P_{n,r_n}}{d\P_{n,\ro_k}}(g)\right|=\frac{1}{n^2}\left|\sum_{1\leq i\leq j\leq k}\sum_{\substack{u<v\\\left(\frac{u-1}{n},\frac{v-1}{n}\right)\in B(i,j,k)}}\left(g_{uv}\log\frac{r_{n,uv}}{\ro_{k,ij}}+(1-g_{uv})\log\frac{1-r_{n,uv}}{1-\ro_{k,ij}}\right)\right|\\
        \leq&\frac{1}{n^2}\sum_{1\leq i\leq j\leq k}\sum_{\substack{u<v\\\left(\frac{u-1}{n},\frac{v-1}{n}\right)\in B(i,j,k)}}\left(\left|\log\frac{r_{n,uv}}{\ro_{k,ij}}\right|+\left|\log\frac{1-r_{n,uv}}{1-\ro_{k,ij}}\right|\right)\\
        \leq&\frac{1}{n^2}\sum_{1\leq i,j\leq k}\sum_{\substack{1\leq u,v\leq n\\\left(\frac{u-1}{n},\frac{v-1}{n}\right)\in B(i,j,k)}}\left(\left|\log\frac{r_{n,uv}}{\ro_{k,ij}}\right|+\left|\log\frac{1-r_{n,uv}}{1-\ro_{k,ij}}\right|\right)
    \end{split}
\end{equation}
The quantity above closely resembles $\|\log r_n-\log\ro_k\|_{L^1}+\|\log(1-r_n)-\log(1-\ro_k)\|_{L^1}$, except that it `over-counts' part of $[0,1]^2$ and ignores other parts. We will make this precise.

For points $(\frac{u-1}{n},\frac{v-1}{n})\in[0,1]^2$ such that $B(u,v,n)\subset B(i,j,k)$ for $i,j$ with $\left(\frac{u-1}{n},\frac{v-1}{n}\right)\in B(i,j,k)$, we have 
\begin{equation}
    \begin{split}
        \frac{1}{n^2}\left|\log\frac{r_{n,uv}}{\ro_{k,ij}}\right|=\|\log r_n-\log\ro_k\|_{L^1(B(u,v,n))}.
    \end{split}
\end{equation}
Now define 
\begin{equation}
    \begin{split}
        C_{i,j,k}:=\left\{(u,v)\in[0,1]^2\left|\,\left(\frac{u-1}{n},\frac{v-1}{n}\right)\in B(i,j,k),\,B(u,v,n)\not\subset B(i,j,k)\right\}\right.,
    \end{split}
\end{equation}
and
\begin{equation}
    \begin{split}
        A_n:=\bigcup_{1\leq i,j\leq k}\bigcup_{(u,v)\in C_{i,j,k}}(B(u,v,n)\cap B(i,j,k)).
    \end{split}
\end{equation}
The set $C_{i,j,k}$ consists of the right-most and top-most points in the square $B(i,j,k)$, and the set $A_n$ is the part of $[0,1]^2$ which is over-counted in \eqref{eq:logdifferenceestimation}. For $(u,v)\in C_{i,j,k}$,
\begin{equation}
    \begin{split}
        \frac{1}{n^2}\left|\log\frac{r_{n,uv}}{\ro_{k,ij}}\right|=&\frac{1}{n^2\lambda(B(u,v,n)\cap B(i,j,k))}\|\log r_n-\log\ro_k\|_{L^1(B(u,v,n)\cap B(i,j,k))}\\
        \leq&k^2\|\log r_n-\log\ro_k\|_{L^1(B(u,v,n)\cap B(i,j,k))},
    \end{split}
\end{equation}
where we use that $\lambda(B(u,v,n)\cap B(i,j,k))=(\frac{i}{k}-\frac{u-1}{n})(\frac{j}{k}-\frac{v-1}{n})\geq\frac{1}{k  ^2n^2}$. Hence,
\begin{equation}
    \begin{split}
        &\frac{1}{n^2}\sum_{1\leq i\leq j\leq k}\sum_{\substack{u<v\\\left(\frac{u-1}{n},\frac{v-1}{n}\right)\in B(i,j,k)}}\left|\log\frac{r_{n,uv}}{\ro_{k,ij}}\right|\\
        =&\frac{1}{n^2}\sum_{1\leq i\leq j\leq k}\sum_{\substack{u<v\\\left(\frac{u-1}{n},\frac{v-1}{n}\right)\not\in C_{i,j,k}}}\left|\log\frac{r_{n,uv}}{\ro_{k,ij}}\right|+\frac{1}{n^2}\sum_{1\leq i\leq j\leq k}\sum_{\substack{u<v\\\left(\frac{u-1}{n},\frac{v-1}{n}\right)\in C_{i,j,k}}}\left|\log\frac{r_{n,uv}}{\ro_{k,ij}}\right|\\
        \leq&\|\log r_n-\log\ro_k\|_{L^1([0,1]^2\setminus A_n)}+k^2\|\log r_n-\log\ro_k\|_{L^1(A_n)}.
    \end{split}
\end{equation}
Since the second part tends to 0 as $n\rightarrow\infty$, there exists an $N_1=N_1(\varepsilon)\geq k$ such that $k^2\|\log r_n-\log\ro_k\|_{L^1(A_n)}<\varepsilon/4$ for all $n\geq N_1$. The argument for the terms $\left|\log\frac{1-r_{n,uv}}{1-\ro_{k,ij}}\right|$ is completely analogous. Then
\begin{equation}
    \begin{split}
        \frac{1}{n^2}\left|\log\frac{d\P_{n,r_n}}{d\P_{n,\ro_k}}(g)\right|<\varepsilon
    \end{split}
\end{equation}
for all $n\geq N_1\geq k\geq N_0$. Substituting this inequality into \eqref{eq:probabilitiesasintegral}, we obtain
\begin{equation}
    \begin{split}
        \left|\frac{1}{n^2}\log \P_{n,r_n}(B(\f,\eta))-\frac{1}{n^2}\log\P_{n,\ro_k}(B(\f,\eta))\right|<\varepsilon.
    \end{split}
\end{equation}
\end{proof}

\subsection{Proof of the upper bound}\label{ss:upperboundproof}
The rest of the proof now follows as in \cite{dhara2019large}. For the sake of completeness, we repeat their argument in this section using the notation introduced in this paper.

Let $\M_n$ be the set of permutations of $n$ objects, and let $G^{\phi_n}_n$ be the graph obtained by relabelling the vertices of $G_n$ with the permutation $\phi_n\in\M_n$. The following lemma shows that there exists a finite subset $\T\subset\M$ such that, for all $n$ large enough, the distribution of $h^{G_n^{\phi_n}}$ can be approximated by $h^{G_n,\tau}$ for some $\tau\in\mathcal{T}$.
\begin{lemma}\label{lemma:relabelling}
Let $r_k,f\in\W$ be of the form \eqref{eq:blockgraphon} with $k\geq1$. Then, for any $\varepsilon>0$, there exists $n_0=n_0(k,\varepsilon)$ and a finite set $\T=\T(k,\varepsilon)$ such that for all $n\geq n_0$ and $\phi_n\in\M_n$, there exists $\tau\in\T$ satisfying
\begin{equation}
    \P_{n,r_k}(d_{\square}(h^{G_n^{\phi_n}},f)\leq\varepsilon)\leq\P_{n,r_k}(d_{\square}(h^{G_n,\tau},f)\leq2\varepsilon)
\end{equation}
\end{lemma}
\begin{proof}
See \cite[Lemma 3.3]{dhara2019large}
\end{proof}

We are now ready to prove Theorem \ref{th:upperbound}. 

\begin{proof}[Proof of Theorem \ref{th:upperbound}]
Fix $\varepsilon>0$ and $\f\in\Wt$. Recall the setup of Theorem \ref{th:upperbound}. Because of Lemma \ref{lemma:blockapprox}, it suffices to prove that there exists an $\eta(\varepsilon)>0$ such that, for all $0<\eta<\eta(\varepsilon)$,
\begin{equation}
    \limsup_{n\rightarrow\infty}\frac{2}{n^2}\log\widetilde{\P}_{n,\ro_k}(\widetilde{B}(\widetilde{f},\eta))=\limsup_{n\rightarrow\infty}\frac{2}{n^2}\log\P_{n,\ro_k}(B(\widetilde{f},\eta))\leq-\inf_{g\in B(\widetilde{f},4\varepsilon)}I_r(g),
\end{equation}
where $k$ is chosen to be sufficiently large according to Lemma \ref{lemma:blockapprox}. Next, we apply a version of Szem\'eredi's regularity lemma, as formulated in \cite[Theorem 3.1]{Chatterjeebook}. This states that, for every $\varepsilon>0$, there exists a finite set $\W(\varepsilon)\subset\W$ such that, for every finite simple graph $G_n$ on $n$ vertices, there exist $\phi_n\in\M_n$ and $h\in\W(\varepsilon)$ such that $h^{G^{\phi_n}_n}\in B(h,\varepsilon)$.

Let $G_n$ be a graph drawn according to $\P_{n,\ro_k}$. Define
\begin{equation}
    B(\W(\varepsilon),\varepsilon):=\{g\in\W\colon\min_{h\in\W(\varepsilon)}d_\square(g,h)\leq\varepsilon\}.
\end{equation}
Then, by the result above,
\begin{equation}
    \begin{split}
        \{h^{G_n}\in B(\widetilde{f},\eta)\}&=\{h^{G_n}\in B(\widetilde{f},\eta)\}\bigcap\left(\bigcup_{\phi_n\in\M_n}\{h^{G^{\phi_n}_n}\in B(\W(\varepsilon),\varepsilon)\}\right)\\
        &=\bigcup_{h\in\W(\varepsilon)}\bigcup_{\phi_n\in\M_n}\{h^{G_n}\in B(\widetilde{f},\varepsilon)\}\cap\{h^{G^{\phi_n}_n}\in B(h,\varepsilon)\}.
    \end{split}
\end{equation}
Since $\W(\varepsilon)$ is a finite set, it suffices to show that
\begin{equation}\label{eq:3.19dhara}
    \limsup_{n\rightarrow\infty}\frac{2}{n^2}\log\P_{n,\ro_k}\left(\bigcup_{\phi_n\in\M_n}\{h^{G_n}\in B(\widetilde{f},\eta)\}\cap\{h^{G^{\phi_n}_n}\in B(h,\varepsilon)\}\right)\leq-\inf_{g\in B(\widetilde{f},4\varepsilon)}I_r(g)
\end{equation}
for all $h\in\W(\varepsilon)$ and $\eta<\varepsilon$. Lemma \ref{lemma:relabelling} yields that the left-hand side of \eqref{eq:3.19dhara} is at most
\begin{equation}
    \begin{split}
        &\limsup_{n\rightarrow\infty}\frac{2}{n^2}\log\P_{n,\ro_k}\left(\bigcup_{\phi_n\in\M_n}\{h^{G_n^{\phi_n}}\in B(h,\varepsilon)\}\right)\\
        &\leq\limsup_{n\rightarrow\infty}\frac{2}{n^2}\max_{\phi_n\in\M_n}\log n!\,\P_{n,\ro_k}\left(h^{G_n^{\phi_n}}\in B(h,\varepsilon)\right)\\
        &\leq\limsup_{n\rightarrow\infty}\frac{2}{n^2}\max_{\tau\in\T}\log\P_{n,\ro_k}\left(h^{G_n^{\phi_n}}\in B(h,2\varepsilon)\right),
    \end{split}
\end{equation}
where we use that $|\M_n|=n!$ and $\log n!=o(n^2)$. Since $\T$ is a finite set, it is enough to show that for each $\tau\in\T$,
\begin{equation}
    \begin{split}
        \limsup_{n\rightarrow\infty}\frac{2}{n^2}\log\P'_{n,\ro_k}\left(h^{G_n,\tau}\in B(h,2\varepsilon)\right)&=\limsup_{n\rightarrow\infty}\frac{2}{n^2}\log\P_{n,\ro_k}\left(h^{G_n}\in B(h^{\tau^{-1}},2\varepsilon)\right)\\
    &\leq-\inf_{g\in B(\widetilde{f},4\varepsilon)}I_r(g).
    \end{split}
\end{equation}
By \cite[Lemma 5.4]{Chatterjeebook}, $B(h_{\tau^{-1}},2\varepsilon)$ is closed in the weak topology. Hence, we can apply the LDP upper bound in the weak topology, stated in \cite[Theorem 5.1]{Chatterjeebook}. Although that the upper bound in the weak topology was proved only for the homogeneous ERRG, the argument generalises to the inhomogeneous ERRG model from this paper. We refer to \cite[Section 4]{Markeringthesis} for more detail. Thus,
\begin{equation}
    \begin{split}
        \limsup_{n\rightarrow\infty}\frac{2}{n^2}\log\P'_{n,\ro_k}\left(h^{G_n}\in B(h^{\tau^{-1}},2\varepsilon)\right)\leq-\inf_{g\in B(h^{\tau^{-1}},2\varepsilon)}I_{\ro_k}(g)\leq-\inf_{g\in B(\widetilde{h},2\varepsilon)}I_{\ro_k}(g).
    \end{split}
\end{equation}
If the event in \eqref{eq:3.19dhara} is empty, then the bound is trivial. In order for the event to be non-empty, we must have that $\ds(\widetilde{h^{G_n}},\widetilde{f})\leq\eta<\varepsilon$ and $\ds(\widetilde{h^{G_n}},\widetilde{h})\leq\varepsilon$, so that $\ds(\widetilde{f},\widetilde{h})\leq2\varepsilon$. Hence, $B(\h,2\varepsilon)\subset B(\f,4\varepsilon)$ and we obtain \eqref{eq:3.19dhara}. The proof is finished by letting $k\rightarrow\infty$ and applying Lemma \ref{lemma:ratecontinuousreference}.
\end{proof}

\section{The rate function for the largest eigenvalue}\label{section:largesteigenvalue}
We sketch the proof of Theorem \ref{th:scalingeigenvalue} as given by Chakrabarty et al. \cite{chakrabarty2020large}. We only give details for the Lemmas \ref{lemma:costsmallperturbation} and \ref{lemma:eigenvaluerateblockapprox}, which are generalizations of results in \cite{chakrabarty2020large}.

Note that when $\beta=C_r$, the infimum in \eqref{eq:rateeigenvalue} is attained at $h=r$ and $\psi_r(C_r)=0$. Take $\beta=C_r+\varepsilon$ with $\varepsilon>0$ small, and assume that the infimum is attained by a graphon of the form $h=r+\Delta_\varepsilon$, where $\Delta_\varepsilon\colon[0,1]^2\to\R$ represents a perturbation of the graphon $r$. 

The proof of Theorem \ref{th:scalingeigenvalue} consists of showing that it suffices to consider $\Delta_\varepsilon$ of the form $\varepsilon\Delta$, identifying the optimal $\Delta$ and calculating $\psi_r$. Chakrabarty et al. \cite{chakrabarty2020large} first prove that $I_r(r+\D_\e)\geq2\e^2$ for any perturbation $\Delta_\e$. The following lemma is an adaptation of \cite[Lemma 3.1 and Lemma 3.2]{chakrabarty2020large} and shows that $I_r(r+\Delta_\varepsilon)$ is of order $\e^2$ in the case $\D_\e=\e\D$.
\begin{lemma}\label{lemma:costsmallperturbation}
If $\Delta_\varepsilon=\varepsilon^\a\Delta$ on some measurable set $B\subset[0,1]^2$, with $\e,\a>0$ and $\D\colon[0,1]^2\to\R$, then the contribution of $B$ to the cost $I_r(r+\Delta_\e)$ is
\begin{equation}
    \int_B\Rs(r(x,y)+\e^\a\D(x,y)\mid r(x,y))=(1+o(1))\frac{1}{2}\varepsilon^{2\alpha}\sint\frac{\Delta^2}{r(1-r)}.
\end{equation}
\end{lemma}
\begin{proof}
Because $\chi(a):=\Rs(a\mid b)$ is analytic for every $a\in[0,1]$ and $b\in(0,1)$, we have
\begin{equation}
    \begin{split}
        \int_B\Rs(r(x,y)+\e^\a\D(x,y)\mid r(x,y))=&\int_B\sum_{n=0}^\infty\frac{1}{n!}\chi^{(n)}(r)\varepsilon^{\a n}\Delta^n=\sum_{n=0}^\infty\frac{\varepsilon^{\a n}}{n!}\int_{B}\chi^{(n)}(r)\Delta^n,
    \end{split}
\end{equation}
where we can swap the integral and sum due to the fact that $\sint|\Rs(f(x,y)\mid r(x,y))|\dx\dy$ is uniformly bounded over all $f\in\W$, since $\log r,\log(1-r)\in L^1([0,1]^2)$. Furthermore, $\chi(b)=\chi'(b)=0$ and $\chi''(b)=\frac{1}{b(1-b)}$. Hence,
\begin{equation}
    \begin{split}
        \int_B\Rs(r(x,y)+\e^\a\D(x,y)\mid r(x,y))=&\frac{1}{2}\varepsilon^{2\a}\int_{B}\frac{\Delta^2}{r(1-r)}+O(\varepsilon^{3\a})=(1+o(1))\frac{1}{2}\varepsilon^{2\a}\int_{B}\frac{\Delta^2}{r(1-r)}.
    \end{split}
\end{equation}
\end{proof}

From the proof of Lemma \ref{lemma:costsmallperturbation} and \cite[Lemma 3.1]{chakrabarty2020large}, it follows that optimal perturbations with $\D_\e$ must satisfy $\|\D_\e\|_{L^2}\asymp\e$, and hence it is desirable to have $\D_\e=\e\D$. Chakrabarty et al. \cite{chakrabarty2020large} argue through block graphon approximants that this is indeed the case. For this argument, we need the following lemma, which shows that block graphons approximate the rate function well.
\begin{lemma}\label{lemma:eigenvaluerateblockapprox}
Let $\overline{r}_n$ and $\overline{f}_n$ be the level-$n$ approximants of $r$ and $f\in\W$, as defined in Section \ref{section:blockapprox}. Then $\limn I_{\overline{r}_n}(\overline{f}_n)=I_r(f)$.
\end{lemma}
\begin{proof}
Let $\varepsilon>0$. By Lemma \ref{lemma:ratecontinuousreference}, $I_{r_N}(f)\rightarrow I_r(f)$ uniformly over all $f\in\W$. Hence, there exists $M_1=M_1(\varepsilon)$ such that for all $N\geq M_1$, $|I_{r_N}(f_N)-I_{r}(f_N)|<\varepsilon/2$. Furthermore, $f_N\rightarrow f$ in $L^2$ and $I_r$ is continuous in the $L^2$-topology in $\W$ by Lemma \ref{lemma:ratecontinuousL2}. Thus, there exists $M_2=M_2(\varepsilon)$ such that $|I_r(f_N)-I_r(f)|<\varepsilon/2$ for all $N\geq M_2$. Choosing $M:=\max\{M_1,M_2\}$ completes the proof.
\end{proof}

The remainder of the proof now follows as in \cite{chakrabarty2020large}. We give a brief summary. Using Lemma \ref{lemma:costsmallperturbation} and exploiting the property that $r$ is of rank 1, Chakrabarty et al. \cite{chakrabarty2020large} show that, in the case $\Delta_\e=\e\D$,
\begin{equation}
    \psi_r(C_r+\e)=(1+o(1))K_r\e^2,
\end{equation}
with
\begin{equation}\label{eq:Krinfimum}
    K_r=\inf_{\substack{\D\colon[0,1]^2\to\R\\\sint r\D=C_r}}\sint\frac{\D^2}{2r(1-r)}.
\end{equation}
Note that the integral in the expression above may be infinite, so the infimum is minimized for some $\Delta$ that counteracts the expression $\frac{1}{r(1-r)}$. Using Lagrange multipliers, we obtain that the infimum is minimized for 
\begin{equation}\label{eq:deltaperturbation}
    \D=\frac{C_r}{B_r}r^2(1-r)
\end{equation}
and
\begin{equation}
    K_r=\frac{C_r^2}{2B_r},
\end{equation}
with $B_r$ as defined in \eqref{eq:Br}. Chakrabarty et al. \cite{chakrabarty2020large} conclude the proof by showing that perturbations $\D_\e$ that are not of the form $\Delta_\e=\e\D$ are asymptotically worse as $\e\rightarrow0$. They do this via block graphons, and use Lemma \ref{lemma:eigenvaluerateblockapprox}.

\bibliographystyle{plain}
\bibliography{bibliografie}

\end{document}